\documentclass[12pt,a4paper,leqno]{amsart}
\usepackage{amsmath,amsfonts}
\usepackage{graphicx}

\newtheorem{theorem}{Theorem}[section]
\newtheorem{proposition}[theorem]{Proposition}

\newtheorem{corollary}[theorem]{Corollary}
\newtheorem{definition}{Definition}[section]

\usepackage{amssymb}
\usepackage{upref}
\usepackage[mathcal]{eucal}

\begin{document}

\title{Integrability via reversibility}
\author{Maciej P. Wojtkowski}
\address{Department of Mathematics Physics and Computer Science
\newline University of Opole
\newline 45-052 Opole, POLAND}
\thanks{The author was partially supported 
by the NCN grant 2011/03/B/ST1/04427} 
\email{mpwojtkowski@math.uni.opole.pl}
\date{\today}

\begin{abstract}
A class of left-invariant second order reversible systems
with functional parameter is introduced which exhibits 
the phenomenon of robust integrability:  
an open and dense subset of the phase space is filled with 
invariant tori carrying quasi-periodic motions, and this 
behavior persists under perturbations within the class.

Real-analytic volume preserving systems are found in this class which have 
positive Lyapunov exponents on an open subset, and the complement
filled with invariant tori.   
\end{abstract}

\maketitle

\section{Introduction}

We study a family of second order dynamical systems on a locally homogeneous 
Riemannian space $M$, modeled on a special solvable Lie group. 
The simplest example is the geodesic flow of a left-invariant
metric. Our class generalizes the examples discovered
by  Butler \cite{Bu1}, and Bolsinov and Taimanov \cite{B-T}. 
In these examples the complete integrability of the geodesic
flow in the tangent bundle $TM$ is accompanied by 
highly non-integrable behavior on an invariant 
submanifold of codimension $n = dim \ M$.
The dynamics there is the suspension of a toral automorphism,
and in \cite{B-T} the hyperbolic automorphism is chosen,
which leads to an Anosov flow. The presence of an Anosov flow
as a subsystem guarantees the positivity of topological entropy.

We show that such a behavior extends to a larger class 
of left-invariant second order systems. This class
is parametrized by a matrix $L$ and a functional parameter $F$: 
a smooth vector field in the unit ball of the Euclidean space.  
We call them {\it $L-F$ systems}. 
The crucial property of $L-F$ systems is their $J$-reversibility,
where $J$ is an appropriate involution of the tangent bundle
$TM$. Let us recall that a system is {\it $J$-reversible}
if the involution $J$ conjugates the forward in time 
dynamics with the backward in time dynamics.

There is a vast literature devoted to $J$-reversible 
systems. The survey paper of Lamb and Roberts \cite{L-R}
contains an extensive bibliography.

In particular there is a version of the KAM theory for the  
$J$-reversible systems. It goes back to Moser \cite{M},
and Sevryuk \cite{S}. In our case we establish
robust integrability: it persists under any small perturbation 
as long as we stay in the family of $L-F$ systems. 
Note that this family is parametrized by an infinite dimensional 
Banach space of vector fields $F$. 
What is notable is that we do not assume volume preservation,
the symmetries imposed on the system force the integrability,
and the occurrence of a finite absolutely continuous invariant
measure. This measure has a density
with respect to the Liouville volume which is only $C^\infty$, 
and typically no real-analytic invariant density exists
(Section 8).

The $J$-reversible KAM theory would give us large subsets of quasi-periodic 
motions for perturbations which are not left-invariant, as long
as they are $J$-reversible. We were unable to check the non-degeneracy
of the unperturbed system required for the application of the KAM
theory. However we conjecture that the non-degeneracy does hold 
for most systems under consideration.

Butler, \cite{Bu2},\cite{Bu3}, used the mechanism  discovered in \cite{B-T} 
to obtain $C^\infty$ examples of integrable volume preserving
systems with positive metric entropy. In our class
we find whole families of real-analytic systems 
with positive metric entropy and a subset 
filled with quasi-periodic motions, open but not dense
(Section 9).

The phenomenon of robust integrability, accompanied by  
positive topological entropy occurs already 
for geodesic flows of linear connections. The generalization
of the geodesic flow of the Levi-Civita connection to more   
general linear connections was discussed in \cite{P-W}.
Such a generalization appears naturally in the study of 
Gaussian thermostats, a class of systems introduced 
by Hoover \cite{H}. The paper of Gallavotti and Ruelle 
\cite{G-R} introduces the Gaussian thermostats in the physical context.
In particular in our class of systems
we find a Gaussian thermostat with the following 
paradoxical behavior (Section 10). For small 
kinetic energy the system is asymptotic to an Anosov flow,
which has a large codimension in the whole phase space. 
For larger values of the kinetic energy the system 
undergoes a drastic change, it becomes integrable:
an open and dense subset in the phase space is filled with 
quasi-periodic motions. The Anosov subsystem is still 
present, but for a subset of initial conditions of full Lebesgue 
measure the solutions stay away from that chaotic subsystem, 
and fill densely invariant tori. 

The author would like to thank  Gabriela 
Tereszkiewicz, Wojciech Czernous and Adam Doliwa
for very useful discussions. Preliminary results of this work
were presented at the seminar of Lai-Sang Young at the Courant 
Institute in November 2014. We are grateful for her insights and
hospitality.

\section{The configuration space}
Our configuration space is a locally homogeneous Riemannian space 
modeled on a special Lie group $G$. We start its description
with the Lie algebra $\mathfrak g$. We assume that $dim \ \mathfrak g = n+1$
and that $\mathfrak g$ contains an $n$ dimensional abelian ideal 
$\mathfrak g_0$. 
We choose an arbitrary scalar product in $\mathfrak g$, and let $b$ denote
a unit vector orthogonal to $\mathfrak g_0$. Since the ideal $\mathfrak g_0$
is assumed to be abelian the Jacobi identity imposes no conditions on the operator 
$L: \mathfrak g_0 \to  \mathfrak g_0$,   
$L = ad_b$.  At this stage we place no restrictions on
the operator $L$. Later on we will consider various special cases.
We endow the Lie group with the left invariant metric determined by our
choice of the scalar product in $\mathfrak g$.

The Levi-Civita connection $\nabla$ on the Riemannian manifold $G$
can be expressed as a tensor on $\mathfrak g$. It can be calculated 
directly, which is done in the fundamental paper of
Milnor \cite{Mi}, where extensive explanations can be found.  
The formulas read 
\begin{equation}
\begin{aligned}
&\nabla_bb=0, \nabla_b\xi = A\xi, 
\ \ \text{for} \ \ \xi \in \mathfrak g_0, \\
&\nabla_\xi b = -S\xi, , \nabla_\xi\zeta = 
\langle S\xi,\zeta\rangle b, \ \ \text{for} \ \  \xi,\zeta \in \mathfrak g_0
\end{aligned}
\end{equation}
where $S= \frac{1}{2}(L+L^*)$ and $A = \frac{1}{2}(L-L^*)$ 
denote the symmetric and skew-symmetric parts of the operator $L$.

The Lie algebra has an important automorphism 
$K: \mathfrak g \to \mathfrak g$, where $-K$ is equal to the 
euclidean reflection in $\mathfrak g_0$. There are only few Lie algebras
with this kind of additional symmetry. It is not difficult to enumerate
all of them. Our class of Lie algebras is singled out by the 
additional property that the invariant subspace of $K$ is a subalgebra. 


Since the automorphism $K$ is orthogonal it generates the automorphism
$\mathcal K$ of the Lie group $G$ which is an isometry. 
This isometry will play crucial role in our discussion.
Let us note that both $K$ and $\mathcal K$ are involutive,
i.e., $K=K^{-1}, \mathcal K = \mathcal K^{-1}$.

The Lie group has the following matrix representation, which we will also 
denote by $G$. It consists of matrices with the block form 
\[
\left[ \begin{array}{rr} 1 & 0 \\ w &  e^{uL} \end{array} \right],
\ \ \ w \ \in \ \mathbb R^n, u \ \in \ \mathbb R.
\]
The abelian ideal  $\mathfrak g_0$ corresponds to 
the normal abelian subgroup $G_0$ consisting of the 
matrices with $u =0$. We obtain convenient coordinates 
$(w,u)\in \mathbb R^{n} \times \mathbb R$ in the group $G$, which is the semi-direct 
product $G_0 \rtimes G_1$, of the abelian 
additive groups  $G_0 =\mathbb R^n$ and $G_1 = \mathbb R$.

The Lie algebra of our matrix group consists of matrices
of the form 
\[
\left[ \begin{array}{rr} 0 & 0 \\ \xi &  \eta L \end{array} \right],
\ \ \ \xi \ \in \ \mathbb R^n, \eta \ \in \ \mathbb R.
\]
We will be using the  linear coordinates 
$(\xi, \eta)\in \mathbb R^{n} \times \mathbb R$ 
in the Lie algebra $\mathfrak g$.

Let us choose a lattice
$\Gamma_0$ in $G_0$ of rank $n$, so that $\Gamma_0\backslash G_0$ 
is an $n$-dimensional torus. Our basic configuration space $M$ 
is the Riemannian manifold  $M= \Gamma_0\backslash G$. 
The manifold $M = \Gamma_0\backslash \left(G_0 \rtimes G_1\right)$ 
is canonically diffeomorphic to
$\left(\Gamma_0\backslash  G_0\right) \times G_1 =
\mathbb T^n\times \mathbb R$. In particular each torus 
in this product has a canonical 
affine structure inherited from the abelian subgroup $G_0$. 
This foliation into tori will play an important role.
We will refer to the leaves of this foliation as {\it toral leaves}. 

In general the Riemannian manifold $M$ is only locally homogeneous,
since the left translations do not factor to the coset space 
$\Gamma_0\backslash G$. However the left translations by elements 
from the subgroup $G_0$ do factor onto $M$, and these isometries 
preserve all the toral leaves, acting on them as translations.

The configuration space $M = \Gamma_0\backslash G$ is non-compact.
In some cases the discrete subgroup $\Gamma _0$ can be enlarged 
to a discrete subgroup $\Gamma$ so that the resulting quotient 
space $N = \Gamma\backslash M$ is compact.
It is so when the one parameter subgroup $e^{uL}, u \in \mathbb R,$
contains an automorphism of $\Gamma_0 \subset \mathbb R^n$.
In general the necessary condition for the existence of a 
quotient with finite volume is the unimodularity of the group $G$,
\cite{Mi}. In our case the group $G$ is unimodular if and only if
$tr \ L =0$.

Let us assume for simplicity that $A=e^L$ is an automorphism of 
the discrete subgroup $\Gamma_0$ (and the torus $\Gamma_0\backslash G_0$).   
The discrete subgroup $\Gamma$ is generated by $\Gamma_0$ and
$A \in G$ ($A =e^L$ can be considered as an element of the subgroup $G_1$:
$w = 0, u =1$.) We will consider such subgroups and the resulting
compact locally homogeneous spaces $N = \Gamma\backslash G$.

\section{Second order left invariant equations on Lie groups}
We will need some general properties of second order equations
on Lie groups, which are invariant under left translations.

A {\it second order equation on a manifold $M$} is a dynamical 
system, i.e., a continuous flow $\Psi^t, t \in \mathbb R$, on the tangent bundle $TM$, such
that for a trajectory $(x(t),v(t)) = \Psi^t(x(0),v(0)), t\in \mathbb R,$ 
where $x(t)\in M$ and 
$v(t)\in T_{x(t)}M$, we have $\frac{dx}{dt} = v(t)$.

Let further $G$ be a Lie group. The tangent bundle $TG$ of the 
group is canonically diffeomorphic to the cartesian product
$G\times \mathfrak g$ by the use of left translations.
With this identification  the derivative of a left translation on $G$,
as a mapping of $TG = G\times \mathfrak g $, is equal to the 
identity in the second factor. 

A second order equation on 
a Lie group is {\it left-invariant} if left translations on the tangent 
bundle $TG = G\times \mathfrak g$ commute with the flow. In other words
the trajectories of the flow are taken to trajectories by left translations:
for any $g\in G$ and any trajectory 
$(x(t),v(t)) = \Psi^t(x(0),v(0)), t\in \mathbb R,$ 
the parametrized curve $(gx(t),v(t))$ in $G\times \mathfrak g$
is also a trajectory of the flow.

The structure of a left-invariant second order equation is described
in the following  
\begin{proposition}
If a continuous flow $\Psi^t$ on the tangent bundle 
$TG = G\times \mathfrak g$ of a Lie group
is a left-invariant second order equation then it 
is a group extension of a flow on $\mathfrak g$, i.e., there is 
a continuous flow $\psi^t:\mathfrak g \to \mathfrak g$ and a 
continuous cocycle
over $\psi^t$ with values in $G$, $h: \mathbb R \times \mathfrak g \to G$,
such that 
\[
\Psi^t(g,v) = (gh(t,v), \psi^t(v)), \ \ 
\text{for} \ \ (g,v)\in G \times \mathfrak g. 
\]
\end{proposition}
\begin{proof}
Let $\pi_{G} : TG \to G$ and 
$\pi_{\mathfrak g} : TG \to \mathfrak g$ be the projections
associated with the identification of $TG$ with $G\times \mathfrak g$.
We define the flow $\psi^t$ in the Lie algebra by
\[
\psi^t(v) = \pi_{\mathfrak g}(\Psi^t(g,v)) \ \ \text{for any} \ \ g \in G.
\]
The mapping $\psi^t:\mathfrak g \to \mathfrak g$ 
is well defined because of left invariance, and hence it must be
a flow. 

Further let 
\[
h(t,v) = g^{-1}\pi_G\left(\Psi^t(g,v)\right)
\ \ \text{for any} \ \ g \in G.
\]
Again since it is well defined the cocyle property can be easily
checked.
\end{proof}
The flow on $\mathfrak g$ described in Proposition 3.1 will 
be referred to as the {\it Euler flow}. This choice of terminology
comes from the Euler equation of the rigid body dynamics, see \cite{A},
Appendix 2.

It follows from this Proposition 3.1 that if the Euler flow has a periodic 
trajectory through $v_0\in \mathfrak g$ of period $T$ then
the map $\Psi^T$ preserves the set $G\times \{v_0\}$ and 
it is equal there to a right translation on $G$.

\section{The dynamical systems: the $L-F$ systems}
The simplest dynamical system that is of interest to us 
is the geodesic flow, which we consider in the tangent bundle $TM$,
rather than the cotangent bundle. Using the Riemannian 
metric we can identify the tangent and cotangent bundles, and
the tangent bundle acquires the natural symplectic form $\omega$.

The isometry $\mathcal K: G \to G$ projects to   
$M = \Gamma_0\backslash G$,  and we denote it again as $\mathcal K$.
This isometry is involutive, i.e., $\mathcal K = \mathcal K^{-1}$.
 
The derivative $D\mathcal K: TM \to TM$ is also involutive and it 
commutes with the geodesic flow 
$\Psi^t: TM \to TM, \ \ t \ \in \mathbb R$,
\[
D\mathcal K \circ \Psi^t = \Psi^{t} \circ D\mathcal K ,    
\ \ \ t \ \in \mathbb R.
\]

Let $\tilde J :TM \to TM$ be the involution which is
identity in the base $M$ (it does not move points in $M$) and it
is equal to minus identity in every tangent space. It is well known that
the geodesic flow is {\it $\tilde J$-reversible}, i.e., 
\[
\tilde J \circ \Psi^t = \Psi^{-t} \circ \tilde J ,    \ \ \ t \ \in \mathbb R.
\]   
The involutions  $\tilde J$ and $D\mathcal K$ commute, hence their
composition is also an involution, and we denote it 
by $J = \tilde J \circ D\mathcal K =
D\mathcal K \circ \tilde J$.  
Since the geodesic flow commutes with $D\mathcal K$ 
it must be also $J$-reversible.

The $J$-reversibility is the fundamental self symmetry
that  is present in our family of dynamical systems 
on $SM$.

The first generalization is from the Levi-Civita connection 
of a left invariant metric on $G$ 
to a left invariant linear connection $\widehat\nabla$ on $G$. 
Such a connection differs from the Levi-Civita connection by
a tensor $B$ in $\mathfrak g$
\[
\widehat \nabla_XY - \nabla_XY  = B(X,Y), \ X,Y \in \mathfrak g.
\]
We assume that the connection  $\widehat\nabla$ has two additional 
properties. Firstly we require that
the parametrization of its geodesics is proportional to the arc length.
We will call such connections {\it para-metric}, since they generalize 
the concept of {\it metric connections}, which have isometric parallel 
transport.  A connection is para-metric if and only if 
$\langle B(X,X),X\rangle = 0$ for every $X\in \mathfrak g$.
The geodesic flow of a para-metric connection $\widehat \nabla$ preserves the unit sphere bundle,
and it will be denoted again by 
$\Psi^t: SM \to SM$.
A discussion of geodesic flows of linear connections
can be found in \cite{P-W}.

Secondly we assume that the symmetric part of the tensor 
$B = \widehat \nabla - \nabla$ is invariant under the 
isometric involution
$K: \mathfrak g \to \mathfrak g$, i.e., for every  $X\in \mathfrak g$
\[
B(KX,KX) = KB(X,X).
\]
If such a property holds we say that the connection is
{\it weakly $\mathcal K$-invariant}. It follows that 
for a left-invariant and weakly $\mathcal K$-invariant connection 
the geodesic flow $\Psi^t$ is again $J$-reversible.

The equations of the geodesic flow can be written as
\[
\frac{dx}{dt} = v,  \ \widehat\nabla_vv = 0, 
\]
where $x(t)\in G$ is a parametrized geodesic.

With the identification of the tangent bundle 
$TG$ with $G\times \mathfrak g$ by left translations
we get there the coordinates 
$(w,u;\xi, \eta)\in \mathbb R^{2n+2}$, 
which were introduced in Section 1. 
In these coordinates the above involutions are given by
\[
\begin{aligned}
&D\mathcal K(w,u;\xi, \eta) =(-w,u;-\xi, \eta), 
\tilde J (w,u;\xi, \eta) =(w,u;-\xi, -\eta),\\
&J(w,u;\xi, \eta) =(-w,u;\xi, -\eta)
\end{aligned}
\]
We need to establish the form of a left-invariant, para-metric,
weakly $\mathcal K$-invariant connection on $G$. Recall that $b\in \mathfrak g$ 
is a unit vector orthogonal to $\mathfrak g_0$. An arbitrary element
$X \in \mathfrak g$ can be written as 
$X = \xi + \eta b, \xi \in \mathfrak g_0$. 
\begin{proposition}
A connection $\widehat \nabla = \nabla + B$ on the group $G$ 
is left-invariant, para-metric and weakly $\mathcal K$-invariant 
if and only if there is a linear operator $C: \mathfrak g_0 \to
\mathfrak g_0$ such that for any $X = \xi + \eta b, \xi \in \mathfrak g_0$. 
\[
B(X,X) = \eta C\xi - \langle C\xi,\xi \rangle b
\]
\end{proposition}
\begin{proof}
Without loss of generality we can assume that $B$ is symmetric.
We have for every $X = \xi + \eta b, \xi \in \mathfrak g_0$,
\[
\begin{aligned}
&0 \equiv\langle B(X,X),X \rangle = 
\langle B(\xi,\xi),\xi \rangle +\eta \langle B(\xi,\xi),b \rangle \\
&+ 2\eta\langle B(\xi,b),\xi \rangle + 2\eta^2\langle B(\xi,b),b \rangle  
+\eta^2\langle B(b,b),\xi \rangle +\eta^3\langle B(b,b),b \rangle.    
\end{aligned}
\]
Since this cubic polynomial in $\eta$ must vanish we get for every 
$\xi \in \mathfrak g_0$  
\[
\begin{aligned}
&\langle B(\xi,\xi),b \rangle 
+ 2\langle B(\xi,b),\xi \rangle =0,
2\langle B(\xi,b),b \rangle  
+\langle B(b,b),\xi \rangle  =0,\\
&\langle B(\xi,\xi),\xi \rangle =0, 
\langle B(b,b),b \rangle =0.     
\end{aligned}
\]
Since the connection is weakly $\mathcal K$-invariant
we must have $KB(b,b) = B(b,b)$ and $KB(\xi,\xi) = B(\xi,\xi)$,
and so we obtain further that 

$B(b,b) = 0$ and $B(\xi,\xi) \perp \mathfrak g_0$.
It follows readily that 
\[
\begin{aligned}
& \langle B(\xi,b),b \rangle  = 0, \\
&B(\xi,\xi) = -2\langle B(\xi,b),\xi \rangle b.
\end{aligned}
\]
Putting $C\xi = 2B(\xi,b)$ we get the desired formula.
\end{proof}
Using Proposition 4.1 and the formulas (1) we obtain by direct calculations 
the equations of the geodesic flow for our special 
connections.
\begin{theorem}
For any para-metric left-invariant and weakly $\mathcal K$-invariant
connection, the geodesic equations in the coordinates 
$(w,u;\xi, \eta)$ in the 
tangent bundle $TG$ are 
\begin{equation}
\begin{aligned}
\frac{d\xi}{dt} = \eta F(\xi),  
\ \ \ & \frac{d\eta}{dt} = -\langle F(\xi),\xi\rangle \\
\frac{dw}{dt} = e^{uL}\xi,
\ \ \ & \frac{du}{dt} = \eta.
\end{aligned}
\end{equation}
where $F(\xi) = L^*\xi -C\xi$, and the matrix $C$ depends only
on the tensor $B =\widehat \nabla- \nabla$, 
namely for $X = \xi +\eta b, \xi \in \mathfrak g_0$
we have $B(X,X) = \eta C\xi -\langle C\xi,\xi\rangle b$.  $\square$
\end{theorem}
Note that in the last Theorem any  matrix $C$ can occur 
with an appropriate choice of the connection $\widehat\nabla$.

The equations (2) factor to the Lie algebra $\mathfrak g$ as 
{\it the Euler equations}, the first line of (2). 
Let us recall that the Euler equations are obtained by 
left translations of velocities along geodesics. It follows from the
left invariance of the connection that the resulting 
curves in $\mathfrak g$ must satisfy the Euler equation, 
see \cite{A}, Appendix 2.

Our final generalization is to replace the linear vector field 
$F(\xi)= F\xi, \xi \in \mathfrak g_0 = \mathbb R^n$, 
in (2) by a general (non-linear) vector field

$F: \mathfrak g_0 \to \mathfrak g_0$. 
Such a dynamical system will be called  an {\it $L-F$ system}. 
In the special case of a left invariant
para-metric connection, when the vector field $F$ is linear,
we will call it a {\it quadratic $L-F$ system}.

To summarize: the operator $L:\mathfrak g_0 \to \mathfrak g_0$
in the euclidean space $\mathfrak g_0$ determines  
the Lie group $G$ with a chosen left invariant metric.
The vector field $F$ determines then the second order equations
(2) in $TG$ which are preserved by any left translation of $G$,
and hence project naturally to $TM$, or $SM$. Note that to define
an $L-F$ system on $SM$ it is enough to have the vector field 
$F$ defined in the closed unit ball of $\mathfrak g_0$.

The crucial property of general $L-F$ systems on $SM$ (or $TM$) 
is their $J$-reversibility.

Actually the $L-F$ systems can be characterized as 
smooth second order left invariant equations on $SM$ which
are $J$-reversible.
\begin{proposition}
Any second order left-invariant equations on $SM$ 
which are also $J$-reversible define an $L-F$ system.
\end{proposition}
\begin{proof}
By Proposition 3.1 the equations factor from $SM$ to
the unit sphere 
$\mathbb S^n =
\{(\xi,\eta)\in \mathfrak g| \xi^2 + \eta^2 =1\}$. 
The resulting Euler system is also $J$-reversible, 
where $J(\xi,\eta) = (\xi,-\eta)$.

Using the coordinates
$(\xi_2,\dots,\xi_n,\eta)$ on the unit sphere in the neighborhood
of the equator $\{\eta =0\}$ we consider 
the smooth function 
\[
U(\xi,\eta) = \frac{d\xi}{dt}.
\]
By the $J$-reversibility we obtain that the function $U$
is odd in the $\eta$ variable, i.e., 
$U(\xi,-\eta) = -U(\xi,\eta)$.
It follows that if $U$ is a smooth function on
the unit sphere then the function 
\[
F(\xi) = \frac{1}{\sqrt{1-\xi^2}}U(\xi,\sqrt{1-\xi^2})
\] 
is well defined and smooth in the closed unit ball in $\mathfrak g_0$,
and $U(\xi,\eta) = \eta F(\xi)$.
\end{proof}

\section{Periodic solutions in $J$-reversible systems}

For $J$-reversible systems there is a very convenient way of 
searching for periodic solutions. It goes so far back 
that it is by now a part of the  mathematical folklore. 
It was formulated explicitly by DeVogelaere \cite{DeV},
and Devaney \cite{D}. We do it again for the convenience of the reader.

Let us consider the subset $\mathcal F$ of fixed points of the involution
$J$, $\mathcal F= \{ p|\ J(p) = p\}$.
\begin{theorem} 
Any trajectory of a $J$-reversible flow $\Psi^t$ which visits
$\mathcal F$ twice must be periodic. Moreover such a trajectory is invariant
under $J$ with the reversal  of time.
\end{theorem}
\begin{proof}
Let $p_0 \in \mathcal F$ be such that there is $t_0> 0$ 
with $\Psi^{t_0}(p_0) \in \mathcal F$. We have
\[
\Psi^{t_0}(p_0) =  J\left(\Psi^{t_0}(p_0)\right) = 
\Psi^{-t_0}J(p_0) = \Psi^{-t_0}(p_0).
\]
Hence the trajectory $p(t) = \Psi^t(p_0)$ is periodic with the period
$T = 2t_0$. Moreover
\[
Jp(t) = J\Psi^{t}(p_0) =  \Psi^{-t}(Jp_0) = 
\Psi^{-t}(p_0) = p(-t).
\]
\end{proof}

The minimal period of a trajectory of $p_0$ in the above proof
is $T= 2t_0$ if and only if $t_0> 0$ is the time of the first return
of $p_0$ to the set $\mathcal F$ of fixed points of $J$.

We apply this principle not to the full $L-F$ system
but only to its factor, the Euler equations
\begin{equation}
\frac{d\xi}{dt} = \eta F(\xi), \ 
\frac{d\eta}{dt} = -\langle F(\xi),\xi\rangle, 
\end{equation}
The involution $J$ descends naturally to the phase space of (3) 
 $(\xi,\eta) \in  \mathfrak g$,
and we denote it again by $J$, $J(\xi,\eta) = (\xi,-\eta)$. 
Clearly the Euler equation is 
$J$-reversible. The set of fixed points of $J$ 
is equal to $\mathcal F =\mathfrak g_0 = 
\{(\xi,\eta)| \eta =0\}$.

Let us consider the open unit ball $B\subset \mathfrak g_0$, with
the boundary, the unit sphere, $S = \partial B$.
\begin{definition}\rm
For a smooth vector field $F = F(\zeta)$ defined on  the closed
unit ball in $\mathfrak g_0$, we say 
that a point $\zeta_0$ in the open unit ball $B$ is {\it escaping} 
if the integral curve $\zeta = \zeta(s)$ of $F$ through $\zeta_0$ 
is defined in a finite closed interval $[s_-,s_+]\ni 0$, $\zeta(0) = \zeta_0$,
$\zeta(s) \in B$ for $s\in (s_-,s_+)$, 
the endpoints $\zeta(s_\pm)$ belong to the unit sphere $S =\partial B$,
and the vector field $F$ is transversal to the unit sphere $S$
at the endpoints of the integral curve, i.e., 
\begin{equation}
\langle F(\zeta(s_-)),\zeta(s_-)\rangle < 0, \ 
\langle F(\zeta(s_+)),\zeta(s_+)\rangle > 0.
\end{equation} 
The integral curve through an escaping point is called
an {\it escaping trajectory}.
\end{definition}

Let us note that for a given smooth vector field $F$ the set of escaping
points in the unit ball is open.

We have the following crucial 
\begin{theorem} 
For any escaping trajectory $\zeta(s), s\in [s_-,s_+]$ 
of the  vector field $f = F(\zeta),
\zeta \in \mathfrak g_0$, 
there are periodic functions $\eta(t)$ and
$u(t)$ with the period $T = 2 t_0$ such that 
$\eta(0) = \eta(t_0)=0, u(0) = s_-, u(t_0) = s_+, s_-\leq u(t) \leq s_+$,
and $(\xi(t),\eta(t))$ is a $T$-periodic solution of the Euler 
equation (3), where $\xi(t) = \zeta(u(t))$. Moreover
$\eta(t)$ is an odd function $\eta(-t) = -\eta(t), t\in \mathbb R,$
and $u(t)$ is an even function 
\[
u(t) = s_-+ \int_0^t\eta(s)ds.
\]
\end{theorem}
\begin{proof}
We restrict the Euler equation (3) to the upper half of the  
unit sphere in $\mathfrak g$,
i.e., to the subset $\{(\zeta,\eta)| \zeta^2 +\eta^2 =1, \eta > 0 \}$.
In this submanifold we can use $\zeta$ as coordinates, and we 
introduce there the time change $\frac{ds}{dt} = \eta$.
After this time change the Euler equation becomes the following 
system
\begin{equation}
\frac{d\zeta}{ds} = F(\zeta).
\end{equation}
Hence the trajectory $\zeta(s), s_-< u < s_+,$ of 
this vector field in the open  ball 
$B \subset \mathfrak g_0$
gives rise to the trajectory $(\xi(t), \eta(t)), t_-< t < t_+,$ 
of the Euler equation in the upper half
of the unit sphere in $\mathfrak g$, with 
$\lim_{t \to t_-}\eta(t) = \lim_{t \to t_+}\eta(t) = 0$. 
It may happen (and it does) that $t_\pm = \pm\infty$. 
This is excluded in our case by the condition (4) which gives us 
$\lim_{t \to t_-}\frac{d\eta}{dt} > 0, \lim_{t \to t_+}\frac{d\eta}{dt} < 0$. 
Hence we have a finite time interval $[t_-,t_+]$ and we can
shift it to the time interval $[0,t_0], t_0 = t_+-t_-$.

By Theorem 5.1 this trajectory
extends to the $T = 2t_0$ periodic  solution of the Euler equation 
$\left(\xi(t),\eta(t)\right), t\in \mathbb R$. 
Putting $u(t) = s_-+ \int_0^t\eta(s)ds$ we 
have
$\xi(t) = \zeta(u(t))$ for $0 \leq t \leq  t_0$.  
Since the periodic solution is invariant under $J$ with the reversal of time, 
we obtain that $\xi(-t) = \xi(t), \eta(-t) = -\eta(t)$. It follows that 
$\int_{-t_0}^{t_0} \eta(s)ds =0$, and consequently 
$u(t)$ is a $T$ periodic function $0\leq u(t) \leq u_0$.
Hence we have also 
$\xi(t) = \zeta(u(t))$ for any $t\in \mathbb R$. 
\end{proof}
In simple terms what this proof reveals is that
an integral curve of $F$ is up to a time change
also a trajectory of the respective Euler flow
in the upper semi-sphere.
For escaping trajectories of $F$ the 
solution of the Euler equation crosses transversally 
the ``equator'' $\{\eta = 0\}$ and in the lower 
semi-sphere it follows the same integral curve of 
$F$ but in the reversed direction. Hence it must be 
a periodic trajectory of the Euler flow.

In the case of a quadratic $L-F$ system we get the following 
\begin{corollary} 
If the linear vector field $F(\xi) = F\xi$ has eigenvalues both with 
positive and negative real parts, then for the Euler equation (3) 
the periodic trajectories fill an open and dense subset of 
$\mathbb S^n \subset \mathfrak g$.

If all the eigenvalues of $F$ are on the imaginary axis 
then for the Euler equation (3)
an open and dense subset of 
$\mathbb S^n \subset \mathfrak g$ is filled with trajectories
which are either periodic or quasi-periodic after an
appropriate smooth time change.
\end{corollary}
\begin{proof}
Let us first note that for linear vector field $F$ the condition (4)
is satisfied on an open and dense subset of 
the unit sphere $S \subset \mathfrak g_0$, unless $F^*=-F$.

If the matrix $F$ has eigenvalues with both positive and negative 
real parts then the set of escaping points in the unit ball 
$B\subset \mathfrak g_0$ is not only open but also dense in the unit ball.
That is so because the instability as $t\to \pm\infty$
forces typical solutions to reach the boundary 
unit sphere at some finite time both in the future and in the past.
By Theorem 5.2 they give rise to the periodic  solutions
of the Euler equations.

If the matrix $F$ has only purely imaginary eigenvalues,
and there is no resonance, then the linear system (5)
has only quasi-periodic solutions. In general there is an open 
set of trajectories which stay in the unit ball for all times,
and an open set of escaping trajectories. The union of these
open sets is dense in the unit sphere. 
The only exception is the case of the skew-symmetric matrix $F^*=-F$,
when there are no escaping trajectories.  
Note that while the solution of (5) is quasi-periodic
the solution of the Euler equation need not be quasi-periodic
since a time change is involved.

In the resonant case we have instability for $t\to \pm\infty$, and
hence there is a dense subset of escaping trajectories, both in 
the cases of purely imaginary nonzero eigenvalues and of the zero eigenvalue. 
By Theorem 5.2 we get an open and dense subset of the unit sphere 
$\mathbb S^n \subset  \mathfrak g$ 
filled with periodic solutions of the Euler equations.
\end{proof}

\section{Robust integrability of $L-F$ systems}
The notion of integrability of a dynamical system
has a long history and several different versions. 
In hamiltonian dynamics the main concept of integrability 
is associated with the Liouville-Arnold theorem, where
families of invariant tori carrying quasi-periodic motions
appear. The issues involved in the general definition 
of integrability in hamiltonian dynamics were explored
by Bogoyavlensky, \cite{Bo1},\cite{Bo2}, and Fasso \cite{F}.

We study a family of reversible, non-hamiltonian 
systems  in which quasi-periodic motions occur robustly.  
We introduce the following working definition
\begin{definition}\rm
We  call a dynamical system {\it semi-integrable} 
if an open subset $U$ of the phase space 
is filled with invariant tori carrying the
quasi-periodic motions. If $U$ is also dense in the
phase space then the dynamical system will be called 
{\it integrable}.
\end{definition}

It turns out that any periodic solution of the Euler equation constructed in 
Theorem 5.2 gives rise to quasi-periodic solutions of the $L-F$ system.
This is the contents of the following theorems.
\begin{theorem}
If a vector field $F$ in 
$\mathfrak g_0$ has escaping points in the unit ball 
then the $L-F$ system on $SM$ is semi-integrable.

If the vector field $F$ has a dense set of escaping points
in the unit ball then  the $L-F$ system on $SM$ is integrable.
\end{theorem}
For quadratic $L-F$ system we can formulate an effective criterion
of integrablity. Moreover the integrablity persists under small 
perturbations in the space of $L-F$ systems.
\begin{theorem}
If for a linear vector field $F = F\xi, \xi \in \mathfrak g_0$ 
the matrix $F$ has eigenvalues with both positive and negative 
real parts then the quadratic $L-F$ system on $TM$ is integrable.
Moreover the invariant tori are common level sets of real-analytic 
first integrals.

Any small perturbation of such a quadratic integrable $L-F$ system in the space
of all $L-F$ systems must be integrable. 
\end{theorem}
The last theorem covers the examples of geodesic flows in Butler,\cite{Bu1},
and Bolsinov and Taimanov, \cite{B-T}.

\begin{corollary}
If the operator $L:\mathfrak g_0 \to \mathfrak g_0$ has eigenvalues with
both positive and negative real parts then for the left-invariant metric on 
$M$ the geodesic flow on $TM$ is integrable.

Moreover  the $(n+1)$-dimensional invariant tori  
are common level sets of  real-analytic 
first integrals in involution. 
\end{corollary}

It seems that within the class of quadratic integrable systems, 
for an open dense subset of operators $F$ the system is  
non-degenerate, in the sense of having a rich family of frequencies on $SM$.
The reversible version of the KAM theory,
developed by Moser, \cite{M}, and Sevryuk, \cite{S},
would be applicable upon the establishment of the non-degeneracy. 
The allowed perturbations would be $J$-reversible second order equations, 
but not necessarily left-invariant.

The calculations required to establish the non-degeneracy 
are cumbersome, and we did not find a satisfactory way to do it.

We will give a joint proof of Theorems 6.1 and 6.2.
\begin{proof}
We will establish that the periodic solution of the Euler equation
in $\mathfrak g$ constructed in Theorem 5.2 is covered by invariant tori 
in $TM$ carrying quasi-periodic solutions. 
To that end, given the $T$-periodic solution $v(t)\in \mathfrak g,
v(t) =(\xi(t),\eta(t))$
of Theorem 5.2, $T= 2t_0$, let us integrate the equations (2)  
of the $L-F$ system.
We get 
\[
\begin{aligned}
& u(t) = a+ a_0(t),  \  a_0(t) = \int_{0}^t \eta(s)ds, \\ 
&w(t) = b +  b_0(t), \  b_0(t) = \int_0^te^{u(s)L}\xi(s)ds.
\end{aligned}
\]
Since the function $\eta$ is odd and periodic we obtain that
$a_0(t)$ is also a periodic even function and assumes values between
$-u_0$ and $u_0$, where $u_0 = \int_{0}^{t_0} \eta(s)ds$. 
Note that although $b_0(t)$ is not periodic, its time derivative
is $T$-periodic. 

Fixing the value $a = u(0)$ we obtain an invariant torus.
Indeed consider the flow $\Psi^t$ defined by the $L-F$ system (3),
which is a group extension flow by Proposition 3.1. 
The mapping $\Psi^T$ takes $M\times\{v(0)\}$ into itself  
and it acts there as a right translation. Since $u(t)$ is periodic
the translation is by an element from $G_0$, and hence it preserves 
all the toral leaves. For $a\in \mathbb R$ and $v\in \mathfrak g$ denote 
by $\mathcal L(a,v)$ the toral leaf 
$\mathcal L(a,v) = 
\{(w,u;\xi,\eta)\in SM| w \ mod \ \Gamma_0, u = a, (\xi,\eta)=v\}$. 
Since $\Psi^T$ takes $\mathcal L(a,v(0))$ into itself, and it acts there 
by the  translation by $c_0 = \int_0^{2t_0}e^{u(s)L}\xi(s)ds$,
then the union of the $n$-dimensional tori 
$\Psi^t\mathcal L(a,v(0)), t \in \mathbb R,$ is an $(n+1)$-dimensional 
torus carrying the suspension flow of the toral translation,
i.e., a quasi-periodic flow.

In the case of a linear vector field $F(\xi) = F\xi$,
for some constant matrix $F$, the real analytic first integrals are given by 
the vector valued function 
\[
\Phi(w,u;\xi,\eta) = e^{-uF}\xi.
\]
These first integrals are functionally independent, hence 
their common level sets must be in general unions of 
submanifolds of dimension $n+1$. 
Clearly the toral leaves $\mathcal L(a,v)$ belong to the level sets,
and so do the $n+1$-dimensional tori constructed above.
Hence in general the level sets must be unions of the 
$(n+1)$-dimensional invariant tori. 
\end{proof}

The geodesic flow on $TM$ is both an $L-F$ system and
hamiltonian but it is an exception. The $L-F$ systems are rarely 
hamiltonian.

To establish the Corollary 6.3 we consider more generally 
hamiltonian systems on $T^*M$ with $G_0$ symmetry.
If the hamiltonian function $H: T^*M \to \mathbb R$ is invariant 
under the action of the torus $G_0/\Gamma_0$ then we get immediately
$n$ first integrals in involution. Indeed, this action in the canonical
variables $(w,u;p_w,p_u)$, associated with the variables $(w,u)$
in $M$, amounts to translations in $w$ with the other variables fixed.
It follows that the hamiltonian $H$ does not depend on the variables $w$
and so the $n$ momenta $p_w$ are first integrals of our
hamiltonian system. Together with $H$ itself we have $n+1$ first 
integrals in involution, and if $H$ is functionally independent of $p_w$
then the Arnold-Liouville Theorem is applicable. In particular
we obtain invariant tori on compact level sets of $H$.

Let us identify the cotangent bundle $T^*M$ with the tangent bundle
$TM$ using the left invariant metric on $M$. This identification is 
the Legendre transform associated with the geodesic 
flow. Since the lagrangian $\mathcal L$ for the geodesic flow is equal to
\[
\mathcal L(w,u;\dot w, \dot u) =
\frac{1}{2}\left (\langle e^{-uL}\dot w,e^{-uL}\dot w \rangle +
\dot u^2\right),
\]
we get immediately that 
\[
p_w = \frac{\partial}{\partial \dot w} \mathcal L = e^{-uL^*}e^{-uL}\dot w =
e^{-uL^*}\xi.
\]
For the geodesic flow these are the same first integrals as 
those in Theorem 6.2.

Let us finally consider more special hamiltonians, 
which are invariant under the action 
of the full group $G$. They can be described by functions $\widetilde H:
\mathfrak g \to \mathbb R$, $\widetilde H = \widetilde H(\xi,\eta)$
and 
$H = \widetilde H(e^{uL^*}p_w,p_u)$.
It is instructive to compare the respective 
hamiltonian equations in the $(w,u;\xi,\eta)$ variables 
in $TM$ with the equations of an $L-F$ system. We have
\[
\begin{aligned}
\frac{d\xi}{dt} = \frac{\partial \widetilde H}{\partial \eta} L^*\xi,  
\ \ \ & \frac{d\eta}{dt} = 
-\langle \frac{\partial \widetilde H}{\partial \xi},L^*\xi\rangle \\
\frac{dw}{dt} = e^{uL}\frac{\partial \widetilde H}{\partial \xi},
\ \ \ & \frac{du}{dt} = \frac{\partial \widetilde H}{\partial \eta}.
\end{aligned}
\]

It transpires that also in the hamiltonian case 
if the function $\widetilde H$ has compact level sets
then the Euler equation in $\mathfrak g$ has an open
and dense set of periodic trajectories.

\section{Compact configuration spaces}
In this section we consider the additional properties of an $L-F$ 
system on the compact phase space, namely the unit sphere bundle $SN$.
where $\Gamma$ is a cocompact lattice in $G$ and 
$N = \Gamma\backslash G$, as discussed in Section 1.

Let us note first that in this  case the group must be unimodular, 
and the operator $L$ has zero trace. Hence either $L$ has  
eigenvalues with both positive and negative real parts,
or all of its eigenvalues are on the imaginary axis.

The integrability and semi-integrability are passed from 
the non-compact phase space $SM$ to the compact phase space $SN$ 
without further assumptions. 
What is new is the appearance of hyperbolic behavior made possible
by the recurrence in the compact phase space. 
For any quadratic $L-F$ system the submanifold 
$\mathcal A \subset SM$ given by the equations $\xi =0, \eta =1$
is diffeomorphic to $M$ and it carries trajectories escaping 
to infinity both in the future and  in the past. 
The projection of $\mathcal A$ to the compact phase space $SN$ is 
diffeomorphic to $N$ and the $L-F$ flow is the suspension 
of the toral automorphism $e^L$. 
If the toral automorphism  
has no eigenvalues on the unit circle then it is an Anosov 
diffeomorphism, and its suspension is an Anosov flow. 
If it has only some eigenvalues outside the unit circle
then we get partially hyperbolic flows as suspensions,
\cite{K-H}. Such flows 
in the compact phase space $SN$ while integrable
have positive topological entropy. The existence of such flows 
was the discovery of Bolsinov and Taimanov, \cite{B-T}.

More generally let us assume that the automorphism  $e^L$ has eigenvalues 
outside of the unit circle. Then for an arbitrary left-invariant 
second order equations on $SN$ (not necessarily an $L-F$ system)
we have the following
\begin{proposition}
If the variable $\eta$ has positive lower (upper) time average on some 
trajectory then this trajectory has negative and positive lower 
(upper) Lyapunov exponents. The dimension of the respective stable 
and unstable subspaces is equal to the dimension of the stable
and unstable subspaces of the automorphism $e^L$. 
\end{proposition}
\begin{proof}
The flow $\Psi^t$ has the structure of a group extension flow over 
the Euler flow as described in Proposition 3.1. 
We fix a trajectory $v(t)= (\xi_0(t),\eta_0(t))$ of the
Euler flow. There is a unique parametrized curve $(w_0(t),u_0(t))\in G,
w_0(0)= 0, u_0(0) = 0$, such that for every $(w,u) \in G$
\[
\Psi^t(w,u;v(0)) = \left(w+e^{uL}w_0(t),u + u_0(t);v(t)\right).
\]
Hence the flow takes the $n$- dimensional toral leaf  
$\mathcal L(u,v(0))$
onto the toral leaf 
$\mathcal L(u+u_0,v(t))$. We have the coordinates $w$ for both 
toral leaves. In these coordinates the restriction of the flow
is defined by $w \mapsto w+e^{uL}w_0(t)$.
Differentiating this mapping with respect to $w$ we obtain 
that $D\Psi^t$ restricted to the $n$-dimensional tangent subspace of
the toral leaf is given by the identity operator in the coordinates $w$. 
However we measure the tangent vectors using 
the left-invariant metric and that means  
\[
||dw||^2 = \left(e^{-(u+u_0)L}dw\right)^2.
\]
Hence we get the exponential growth (decay) as long as
$u_0$ has linear growth (decay). Since 
$\frac{du_0}{dt} = \eta_0$, 
the assumption of the positive time average 
of $\eta_0$ leads to the exponential growth (decay) of $||dw||^2$ 
in the unstable (stable) subspace of the operator $L$. 
\end{proof}
We have not chosen any invariant measure so the Lyapunov exponents 
need to be understood as the upper (or lower) limits. 
It follows from the proof that assuming negative lower (upper)
time average we will arrive at the same conclusion,
with the reversal of the number of positive and negative 
Lyapunov exponents.

The phenomenon described in Proposition 7.1 appears explicitly
in the work of Butler and Paternain on special magnetic flows
on $SOL$, \cite{Bu-P}.

Let us consider first integrals of a quadratic $L-F$ flow $\Psi^t$  
on the compact phase space $SN$. The real-analytic first integrals
of Theorem 6.1 can be considered as real analytic multi-valued first 
integrals.  We will establish that for a residual 
subset of integrable quadratic $L-F$ systems on $SN$
there are no single valued real-analytic first
integrals constant on the toral leaves. 
If an integrable system is non-degenerate (i.e., it has 
a rich family of frequencies for the quasi-periodic motions on 
the tori) then by necessity any continuous first integral must be
constant on the toral leaves. Since we have not established
the non-degeneracy we need to include the constancy into the assumptions.

Let us recall that by Theorem 6.1 quadratic $L-F$ flows on $SN$ 
for which the matrix $F$ has eigenvalues with 
positive and negative real parts are integrable, and they form 
an open subset in the space of all quadratic $L-F$ flows. 
The complement of this subset has an interior filled with flows 
asymptotic to the suspension of the toral automorphism $e^L$.   
\begin{proposition}
For a residual subset of integrable quadratic $L-F$ flows 
on $SN$  there are no single valued real-analytic 
first integrals constant on the 
$n$-dimensional toral leaves, i.e., functions of $(u;\xi,\eta)$.
\end{proposition} 
\begin{proof}
We consider only diagonalizable operators $F$,  with 
all different eigenvalues
$(\lambda_1,\lambda_2,\dots,\lambda_n)$,
at least some of them with positive and 
some of them with negative real parts. 
This gives us an open and dense subset of integrable quadratic 
$L-F$ systems. If there is a real-analytic first integral $R$ 
for such a system on $SN$ then it can be lifted to the first 
integral of the system on $SM$, which we denote also by $R$. 
Now $R$ can be considered as a function of $(u;\xi,\eta)$. 
We restrict our attention to the unit ball $B$ in $\mathfrak g_0$ 
(the ``$\xi$-space''), and we put $\eta = \sqrt{1-\xi^2}$.
This gives us a real-analytic function 

$\widetilde R(u;\xi), (u;\xi)\in \mathbb R\times B,$ which is 
also a first integral of the system.

We further consider the time change 
$\frac{ds}{dt} =\eta$ in our domain $\mathbb R\times B$ 
The equations of the system in the variables $(u;\xi)$ are 
\begin{equation}
\frac{du}{ds} = 1, \ \frac{d\xi}{ds} = F\xi
\end{equation}
The function $\widetilde R$ is a first integral of 
(6) so that  its time derivative is identically zero
\[
\frac{\partial \widetilde R}{\partial u}  + 
\frac{\partial \widetilde R}{\partial \xi}F\xi  \equiv 0. 
\]
It follows that for the real analytic function 
$P(u;\zeta) =  \widetilde R(u;e^{uF}\zeta)$ 
defined in the vicinity of the origin we have 
\[
\frac{\partial P}{\partial u}  = 
\frac{\partial \widetilde R}{\partial u}  + 
\frac{\partial \widetilde R}{\partial \xi}F e^{uF}\zeta
\equiv 0. 
\]
Hence $P$ is the function of $\zeta$ alone defined in the vicinity
of the origin in $\mathbb R^n$ such that   
$\widetilde R(u;\xi) = P(e^{-uF}\xi)$. 
Now we use the fact that the function $\widetilde R(u;\xi) = P(e^{-uF}\xi)$
must be actually a periodic function of $u$ with the period $1$, since 
$\widetilde R$ is lifted from $SN$. Hence 
\begin{equation}
P(e^{-F}\xi) = P(\xi), \ \ \text{for all} \ \ \xi \in \mathbb R^n.
\end{equation}
It goes back to Poincare that such a relation is 
impossible if the eigenvalues of $F$ are non-resonant. More precisely, if  
there is no integer vector $(k_1,k_2,\dots,k_n)$, with non-negative 
entries, such that 

$k_1\lambda_1 +\dots +k_n\lambda_n =0$ then the relation (7)
can hold only for a constant function.

The family of $L-F$ systems with non-resonant
eigenvalues form a residual subset of the open set of 
integrable quadratic $L-F$ systems introduced at the 
beginning of the proof.  
\end{proof}

\section{Invariant measures}

In this section we discuss smooth invariant measures for the 
$L-F$ systems, for the non-compact $SM$ and compact $SN$ phase spaces. 
First of all we consider the Lebesgue measure $dw\ du$ in $G$,
which is the right-invariant Haar measure. It is also invariant
under the left translations by elements from $\Gamma_0$. So it projects
into a $\sigma$-finite measure $\mu$ on $M = \Gamma_0\backslash G$. 
The product of $\mu$ by the standard Lebesgue measure 
in $\mathbb S^n \subset \mathfrak g$ (or $\mathfrak g$) 
will be referred to as the Lebesgue measure in $SM = M\times \mathbb S^n$
(or $TM$) and denoted by $\nu$.
Once the Lebesgue measure $\nu$ is chosen we call 
a function $\rho\geq 0$ an {\it invariant density} of a dynamical system 
if the measure $\rho\nu$ is preserved by the dynamical system.
\begin{proposition}
If the vector field $F$ has constant divergence in the unit ball 
of $\mathfrak g_0$ then $\rho = exp(-u\ div \ F)$ is an invariant density 
for  the $L-F$ system on $SM$.
\end{proposition} 
\begin{proof}
Since the $L-F$ system defined in $TM = M\times \mathfrak g$ 
preserves all the sphere bundles it suffices to check the invariance
of the measure in $TM$. It can be easily seen that 
the divergence of the $L-F$ system in $TM$ is equal to 
$\eta \ div \ F$. Since
\[
\frac{d}{dt} \rho = -\eta \ div \ F,
\]
the claim is  proven.
\end{proof}

In the case of a compact  phase space $SN$
the group $G$ must be unimodular, i.e., $tr \ L = 0$. 
For unimodular Lie groups the left Haar measure is equal 
to the right Haar measure, and the Lebesgue measures $\mu$ 
and $\nu$ project as finite measures to $N$ and $SN$, respectively. 
We will denote the resulting measures again as $\mu$ and $\nu$.

It follows from the above Proposition that if $div \ F = 0$ then 
the $L-F$ system preserves the Lebesgue measure $\nu$ in $SN$.
Moreover the Euler flow of the $L-F$ system preserves the Lebesgue
measure on $\mathbb S^n \subset \mathfrak g$.

Conversely, if an $L-F$ system has a (real-analytic)  
invariant density $\rho$ in $SN$ then the Euler flow 
in $\mathbb S^n\subset \mathfrak g$ 
has a (real-analytic) invariant density.
Indeed, we obtain the density of the projected measure by 
integrating out the $w,u$ coordinates. It preserves the smoothness,
and the real-analyticity of the density.

For the quadratic $L-F$ systems the Euler flow cannot have
a smooth invariant density which is  positive at the fixed point
$\xi =0$, unless $div \ F = 0$. However there may be smooth,
or even real-analytic, invariant densities which vanish at the 
fixed point, even if $div \ F \neq 0$.   

\begin{proposition}
For a residual subset in the space of quadratic 
$L-F$ systems in $SN$ there are no real-analytic invariant densities.

For the open family of integrable quadratic 
$L-F$ systems in $SN$ with the matrix $F$ having  all 
different real eigenvalues, there are $C^\infty$ invariant densities. Moreover 
there is a dense subset in this family with real-analytic invariant densities.
\end{proposition}
\begin{proof}
It follows from the above considerations that 
if an $L-F$ system in $SN$ has a real analytic invariant density 
then the field $F$ has a real analytic invariant
density $\rho = \rho(\xi)$ in the unit ball in $\mathfrak g_0$.
Hence we have 
\[
div \left(\rho F\right) = 0.
\]
For a diagonalizable matrix $F$ we can analyze the last equation in the
coordinates in which $F$ is diagonal. This may require a linear complex 
change of coordinates, but that is not a problem in calculations involving
Taylor expansions. Assuming that $F$ is already diagonal with 
eigenvalues $\lambda_1,\lambda_2,\dots, \lambda_n$
we get 
\begin{equation} 
\sum_{i=1}^n\lambda_i\xi_i\frac{\partial \rho}{\partial \xi_i} =
-  tr F \rho
\end{equation}
This relation can hold for a function real-analytic 
in the neighborhood of the origin $\xi =0$ if and only if there 
are monomials which satisfy it. For $\rho(\xi) = \prod_{i=1}^n\xi^{r_i}$
we get substituting into (8)
\begin{equation}
\sum_{i=1}^n\lambda_i(r_i+1) = 0.
\end{equation}
The set of matrices $F$ for which (9) does not hold 
for any vector $(r_1,r_2,\dots,r_n)$ with natural entries
is a residual subset in the space of all matrices.
The first part of the Theorem is proven.

To prove the second part consider the set of matrices 
with all different real eigenvalues, at least one positive 
and one negative. It is an open set of matrices. 
Its subset where the eigenvalues satisfy 
(9) for some vector $(r_1,r_2,\dots,r_n)$ with natural even entries
is dense.  For such systems the function $\rho(\xi) =
\prod_{i=1}^n\xi^{r_i} \geq 0$ is an invariant density of the 
Euler flow, and also the full $L-F$ system in $SN$.

To get a $C^\infty$ invariant density we choose a vector
$(r_1,r_2,\dots,r_n)$ with positive entries satisfying (9) 
and consider the function $\widetilde \rho(\xi) =
\prod_{i=1}^n\left|\xi\right|^{r_i} \geq 0$. The function satisfies 
(8) in the open dense subset where it is positive, but
it is not in general a  smooth function.
We consider the first integral of the Euler 
flow $f(\xi) = \prod_{i=1}^n\left|\xi\right|^{r_i+1}$ and
we get a $C^\infty$ invariant density 
$\rho(\xi) = exp(-f(\xi)^{-1})\widetilde \rho(\xi)$.
\end{proof} 
The $C^\infty$ invariant densities in the proof are similar to the 
$C^\infty$ first integrals of Butler, \cite{Bu1}.

\section{Coexistence of integrablility and positive metric entropy}

The results of Sections 6 and 7  may be used to find real-analytic
volume preserving $L-F$ flows with both quasi-periodic motions in
an open subset and positive metric entropy. 
Such examples were constructed by Butler
in the $C^\infty$ class, \cite{Bu2},\cite{Bu3}. The recent survey 
of Chen, Hu and Pesin, \cite{C-H-P} describes other kinds of 
coexistence phenomenae.

Let $F$ have purely imaginary all different eigenvalues,
but $F^* \neq -F$. For such a linear vector field the open unit ball
is not invariant, hence there is an open subset of escaping points.
More precisely  the whole open unit ball is the closure of  
the union of two 
open subsets $V_e$ and $V_b$. The open subset $V_e$ contains only escaping
points, and the open subset $V_b$ contains points whose trajectories 
have compact closures in the unit ball. Actually these compact closures
are invariant tori of the linear system defined by the matrix $F$.
Note also that the linear vector field has zero divergence, 
so the $L-F$ flow preserves the Lebesgue measure.

These considerations apply already for 
$n=2$. If $F$ has purely imaginary eigenvalues,
but $F^* \neq -F$, then the open unit disk 
contains an open ellipse filled with elliptical integral curves,
The rest of the open unit disk is filled with escaping trajectories.

The open set of escaping trajectories $V_e$ gives rise by Theorem 6.1
to semi-integrability of the $L-F$ flow. 
If the matrix $L$ has eigenvalues with positive (and negative)
real parts then the open set of bounded trajectories $V_b$
leads by Proposition 7.1 to positive
Lyapunov exponents. By Pesin formula, \cite{K-H}, we obtain also positive 
metric entropy with respect to the invariant Lebesgue measure in $SN$.
We get 
\begin{theorem}
There are  volume-preserving quadratic $L-F$ systems on $SN$ which
are semi-integrable and have positive metric entropy.
\end{theorem}    
Let us note that these properties will hold also for  small 
divergence free perturbations of $F$ as long as we guarantee 
that the perturbed vector field in $\mathfrak g$ still has a positive 
Lebesgue measure of quasi-periodic motions near the origin $\xi=0$. 
This can be achieved for instance when the linear vector field 
$F$  and the perturbation are hamiltionian and  satisfy
the non-degeneracy conditions of the KAM theory, \cite{A}, Appendix 8.

Further let us consider the modified linear vector field $F_1 = F -\epsilon I$.
We obtain for  $\epsilon >0$ the asymptotic stability of the 
origin $\xi =0$. At the same time for small $\epsilon$ the restriction of 
the vector field $F_1$ to the unit ball still has an open set of 
escaping trajectories. 
\begin{theorem}
There is an open set of quadratic $L-F$ systems on $SN$ which
are semi-integrable and the interior of the complement 
to the set of invariant tori is filled with orbits which are
asymptotic as $t \to +\infty \ (-\infty)$ to the attractor 
(repellor) carrying 
the suspension  of the toral automorphism $e^L\ (e^{-L})$.
\end{theorem}

We conjecture that the following is true.

{\bf Conjecture.}
{\it There are real-analytic divergence free vector fields in the closed unit ball
with an open and dense subset filled with escaping trajectories,
and with the complement of positive Lebesgue measure,
filled with trajectories defined for all times and 
which are contained in the interior of the unit ball.}

In the Arnold diffusion scenario a generic hamiltonian perturbation 
of an integrable hamiltonian system has a nowhere dense 
subset of invariant tori of positive Lebesgue measure. In the complement
almost all orbits move away unboundedly both in the future and in the past. 
However this scenario was not so far rigorously established in 
any real-analytic examples.

If this conjecture holds then the respective real-analytic $L-F$ system 
would be integrable and of positive metric entropy.

\section{Gaussian thermostats}

A Gaussian thermostat is defined  by a Riemannian metric on a manifold $M$
and a vector field $E$. The trajectories of the Gaussian thermostat
satisfy the following ordinary differential equations in the tangent
bundle $TM$,\cite{P-W}. 
\begin{equation}
\frac{d}{dt}x = v, 
\nabla_vv = E-\frac{\langle E,v\rangle}{\langle v,v\rangle}v,
\end{equation}
where $x = x(t)\in M$ is a parametrized curve in $M$.

By the force of these equations the ``kinetic energy'' $v^2$ is constant. 
Fixing the value of this constant $k = v^2$, and introducing the auxiliary vector field $F_k = \frac{1}{k}E$ we can rewrite the equations
(10) as
\begin{equation}
\frac{d}{dt}x = v, 
\nabla_vv = v^2F_k-\langle F_k,v\rangle v.
\end{equation}
It was observed in \cite{P-W} that the equations (11) describe 
geodesics of a special metric connection $\widetilde \nabla$ defined by the field $F_k$ 
\[
\widetilde\nabla_XY = \nabla_XY - \langle X,Y\rangle F_k + \langle Y,F_k\rangle X,
\]
where $X,Y$ are arbitrary smooth vector fields on $M$.
This connection is not symmetric, it is the unique metric 
connection with the torsion $T(X,Y) =\langle Y,F_k\rangle X 
-\langle X,F_k\rangle Y$, \cite{P-W}. 

Note that the equations of a Gaussian thermostat define significantly
different flows for different values of $k =v^2$, while the equations
(11) scale so that for different values of $v^2$ we get the same flow
up to a change of time by a constant factor. To get complete understanding
of the dynamics of (10) we need to consider the dynamics of the geodesic flow
of the metric connection (11) for the whole family of vector fields 
$F_k = \frac{1}{k} E, k >0$.

If the manifold $M$ is the locally homogeneous 
space of Section 1  and the vector field $E$ is left invariant, 
then we obtain a family of left-invariant para-metric connections.
To get a weakly $\mathcal K$-invariant connection we have to have 
$KE = E$, i.e., $E$ must be orthogonal to $\mathfrak g_0$. 
The equations of geodesics for these connections give us a family
of quadratic $L-F$ systems.  
Assuming that $E^2 =1$, we get that in this case the respective vector 
field in $\mathfrak g_0$ is equal to
\[
F(\xi) = L^*\xi -\frac{1}{k}\xi.
\]
We obtain the following corollary of Theorem 6.2 and
the considerations leading to Theorem 9.2.

\begin{corollary}
Let $r_{min}$ and $r_{max}$  be the smallest and the largest 
real parts of the eigenvalues of $L$.

If $r_{min} < r_{max}$  then for 
$k =v^2$ such that $r_{min} < \frac{1}{k} < r_{max}$ the Gaussian thermostat
(10) in $M$ is integrable.

For $k =v^2$ such that 
$\frac{1}{k} > r_{max}$ the trajectories of the Gaussian 
thermostat (10) in the compact space $N$ are asymptotic 
to a subsystem, which is 
the suspension of the automorphism $e^L$.

\end{corollary}
In particular if $r_{min} < 0< r_{max}$ then for $0 < k < \frac{1}{r_{max}}$
the Gaussian thermostat (10), for the field $E$ orthogonal to $\mathfrak g_0$,
has a global attractor carrying the suspension of $e^L$ 
(which can be an Anosov flow).
However for $k > \frac{1}{r_{max}}$ it is integrable.
This is a somewhat paradoxical behavior: strong mixing properties at
low kinetic energy, and quasi-periodic motions for high kinetic energy.
Admittedly the strong mixing properties occur at the asymptotic
submanifold of half the dimension of the phase space. And
this subsystem is present for all values of the kinetic energy,
but it is not an attractor or a repellor in the integrable case.

\section{A family of Riemannian metrics with integrable geodesic flows}

The Gaussian thermostats can be associated with Weyl connections,
\cite{W1},\cite{W2}.  The left invariant vector field $E$ of Section 10 
is a gradient of a function on the group $G$, which factors to 
$M =\Gamma_0\backslash G$. Indeed we have 
$E = \frac{\partial}{\partial u} = \nabla u$. 
The Weyl connection defined by gradient vector fields are 
Levi-Civita connections of modified metrics, namely for the 
vector field $F_k = \frac{1}{k}E$ we need to 
multiply the initial left invariant metric by the function 
$e^{\frac{2u}{k}}$. 

Let for simplicity $L$ be diagonal with eigenvalues 
$\lambda_1,\lambda_2,\dots,\lambda_n$. The left invariant 
metric on $G$ is given by 
\[
ds^2 = e^{-2u\lambda_1}dw_1^2 +  e^{-2u\lambda_2}dw_2^2 + \dots
 +e^{-2u\lambda_n}dw_n^2 + du^2.
\]
Introducing the new variables  
\[
x_0 =e^{\frac{u}{k}}, x_i = \frac{1}{k} w_i, i = 1, \dots, n,
\]
we obtain 
\[
e^{\frac{2u}{k}}ds^2 = k^2\left(x_0^{2\tau_1}dx_1^2 +  x_0^{2\tau_2}dx_2^2 + \dots
 +x_0^{2\tau_n}dx_n^2 + dx_0^2\right), 
\]
where $\tau_i = 1 - k\lambda_i, i = 1, \dots, n$.
It was calculated in  \cite{T-W} that the Weyl sectional curvatures 
are non-positive in this case if and only if $\tau_i \geq 1$ or 
$\tau_i =0$ for $i = 1, \dots, n$.
The sign of  sectional curvatures of the last metric must be the same as  
for the Weyl connection (\cite{P-W}).

Considering the geodesic flow of the Riemannian  metric 
on $M = \mathbb T^n \times \mathbb R_+ = \{(x;x_0)| x \ mod \ \Gamma_0, x_0>0\}$, 
where $\Gamma_0$ is a lattice in $\mathbb R^n$, we conclude that 
it is integrable if and only if  there are exponents $\tau_i$ 
of opposite signs. It follows from the considerations of Section 10,
but this phenomenon can be greatly generalized. Note that the metric 
is not complete since the variable $x_0$ is assumed positive.
Nevertheless in the integrable case geodesics from an open and dense 
subset can be extended indefinitely since they do not leave a compact 
subset of $M$. 
If the conditions of integrability are not satisfied then 
almost all geodesics leave every compact subset of $M$.

For an open interval $(a,b)\subset \mathbb R$, finite or infinite, 
let us consider a Riemannian metric on 

$M = \mathbb T^n \times (a,b) = \{(x;x_0)| x \ mod \ \Gamma_0, a < x_0 < b\}$, 
\[
ds^2 = \alpha^2_1dx_1^2 +  \alpha^2_2dx_2^2 + \dots
 +\alpha^2_ndx_n^2 + dx_0^2, 
\]
where all the functions $\alpha_i$ are positive functions of the variable
$x_0$ alone, defined on the interval $(a,b)$.

\begin{proposition}
If  $\lim_{x_0\to a}\alpha_i(x_0) =0$ and  $\lim_{x_0\to b}\alpha_l(x_0) =0$ 
for some $i$ and $l$ then the geodesic flow is integrable.
\end{proposition}
\begin{proof}
Let us pass to the hamiltonian formulation. We get 
the momenta 
\[
p_0 = \dot x_0,\ p_j = \alpha^{2}_j \dot x_j, j = 1, \dots, n, 
\]
and the hamiltonian
\[
H = \frac{1}{2}\left(p_0^2 + \sum_{j=1}^{n}\alpha^{-2}_jp^2_j\right).
\]  
We have the obvious $n$ first integrals $p_1, \dots, p_n$. Together 
with $H$ we get $n+1$ first integrals in involution. Under the 
assumption that $p_i \neq 0$ and $p_l\neq 0$ the common level set of the
first integrals must be compact in the variable $x_0$, hence altogether 
compact. Indeed, we have
\[
H \geq \alpha^{-2}_lp^2_i +\alpha^{-2}_lp^2_l\to +\infty 
\ \ \text{as} \ \  x_0 \to a,b.
\]
By the Arnold-Liouville Theorem all regular compact level sets are tori
carrying the quasi-periodic motions.
\end{proof}
This last Proposition takes us away from homogeneous systems,
now we need only the $\mathbb T^n$ symmetry.


\begin{thebibliography}{C-H-P}

\bibitem[A]{A}
Arnold, V.I.: Mathematical Methods of Classical Mechanics,
Springer, New York, (1989)

\bibitem[C-H-P]{C-H-P}Chen, J.; Hu, H.; Pesin, Y.: 
The essential coexistence phenomenon in dynamics. 
Dyn. Syst. 28, 453–-472 (2013) 

\bibitem[Bo1]{Bo1} O. I. Bogoyavlenskij, O.I.: 
A concept of integrability of dynamical systems. 
C. R. Math. Rep. Acad. Sci. Canada 18(4), 163–-168 (1996)

\bibitem[Bo2]{Bo2} O. I. Bogoyavlenskij, O.I.: 
Extended integrability and bi-Hamiltonian systems. 
Commun. Math. Phys. 196(1), 19–-51 (1998) 

\bibitem[Bu1]{Bu1} Butler, L.: 
A new class of homogeneous manifolds with Liouville-integrable geodesic flows. 
C. R. Math. Acad. Sci. Soc. R. Can. 21,  127–-131 (1999)

\bibitem[Bu2]{Bu2} Butler, L.: 
Integrable hamiltonian flows with positive Lebesgue-measure entropy. 
Ergodic Theory Dyn. Syst. 23, 1671–-1690 (2003)

\bibitem[Bu3]{Bu3} Butler, L.: An integrable, volume-preserving flow 
on S2×S3 with positive Lebesgue-measure entropy. 
Topological methods in the theory of integrable systems, 81–-87, 
Camb. Sci. Publ., Cambridge (2006)


\bibitem[Bu-P]{Bu-P} Butler, L.; Paternain, G. P.: 
Magnetic flows on Sol-manifolds: dynamical and symplectic aspects. 
Comm. Math. Phys. 284, 187–-202 (2008). 


\bibitem[B-T]{B-T} Bolsinov, A.V., Taimanov, I.A.: 
Integrable geodesic flows with positive topological entropy. 
Invent. Math. 140, 639–-650 (2000)




\bibitem[D]{D}Devaney, R.L.: Reversible diffeomorphisms and flows. 
Trans. Am. Math. Soc. 218, 89--113 (1976)

\bibitem[DeV]{DeV} DeVogelaere, R.:  On the structure of periodic 
solutions of conservative systems, with applications.
In: Lefschetz, S. (Ed.), Contribution to the Theory
of Nonlinear Oscillations, vol. 4. Princeton Univ. Press,  53--84 (1958)


\bibitem[F]{F} Fasso, F.: Quasi-periodicity of motions and complete 
integrability of Hamiltonian systems. Ergod. Th. Dynam. Sys. 18(6), 
1349–-1362 (1998)


\bibitem[G-R]{G-R} Gallavotti, G., Ruelle, D.: 
SRB states and nonequilibrium statistical mechanics close to 
equilibrium. Commun. Math. Phys. 190, 279–-285 (1997)


\bibitem[H]{H} Hoover, W.G.: Molecular Dynamics,
Lecture Notes in Phys. 258, Springer (1986)



\bibitem[L-R]{L-R} Lamb, J. S. W.; Roberts, J. A. G.: 
Time-reversal symmetry in dynamical systems: a survey. 
 Phys. D 112, 1–-39 (1998). 


\bibitem[K-H]{K-H} Katok, A.; Hasselblatt, B.: 
Introduction to the modern theory of dynamical systems. 
Cambridge Univ. Press, (1995)


\bibitem[Mi]{Mi} Milnor, J.: Curvatures of left invariant metrics on Lie 
groups. Advances in Math. 21 , no. 3, 293–-329 (1976)



\bibitem[M]{M} Moser, J.: 
Convergent series expansions for quasi-periodic motions, 
Math. Ann. 169,  136--176 (1967)





\bibitem[P-W]{P-W}Przytycki, P.; Wojtkowski, M. P.: 
Gaussian thermostats as geodesic flows of nonsymmetric linear connections. 
Comm. Math. Phys. 277, 759–-769 (2008) 




\bibitem[S]{S} Sevryuk, M.B.: The finite dimensional reversible
KAM theory. Physica D 112, 132--147 (1998)




\bibitem[T-W]{T-W} Tereszkiewicz, G; Wojtkowski, M.P.:
Homogenous Weyl connections of non-positive curvature.
arXiv: 1506.08176 (2015)


\bibitem[W1]{W1}
Wojtkowski, M.P.: W-flows on Weyl manifolds and Gaussian thermostats. 
J. Math. Pures
Appl. 79(10), 953-–974 (2000)
\bibitem[W2]{W2}
Wojtkowski, M.P.: Weyl manifolds and Gausssian thermostats. 
In: Proc. ICM Beijing 2002, Beijing:
Higher Education Press, 2003, pp. 3511–-3523


\end{thebibliography}
\end{document}